\documentclass{amsart}
\usepackage{amssymb,latexsym,amsmath,bm,mathtools,listings,xcolor,caption,graphicx,booktabs,array,color,amsthm,multirow,setspace,indentfirst,subfigure,mathrsfs}
\usepackage{multirow}
\usepackage{makecell}
\usepackage[mathscr]{eucal}
\numberwithin{equation}{section}
\newtheorem{lemma}{Lemma}[section]
\newtheorem{definition}{Definition}[section]
\newtheorem{theorem}{Theorem}[section]
\newtheorem{assumption}{Assumption}
\newtheorem{proposition}{Proposition}[section]

\def\bp{\boldsymbol{p}}
\def\bx{\boldsymbol{x}}
\def\bf{\boldsymbol{f}}
\def\bX{\mathbf{X}}
\def\d{\mathrm{d}}
\def\bdiv{\mathrm{div}}
\def\divp{\mathrm{div}\boldsymbol{p}}

\begin{document}
\title{Error Analysis of Mixed Residual Methods for Elliptic Equations}
\author{Kai Gu, Peng Fang, Zhiwei Sun, Rui Du}
\address{School of Mathematical Science, Soochow University, Suzhou, P. R. China}
\email[Kai Gu]{1907404012@stu.suda.edu.cn}
\email[Peng Fang]{pengfang010809@163.com}
\email[Zhiwei Sun]{20194007008@stu.suda.edu.cn}
\email[Rui Du]{durui@suda.edu.cn}
%\email[Tianxiao Ye]{913551286@qq.com}
%\email[Qingran Wang]{15851916117@163.com
%}
\date{May 2023}
\keywords{MIM, DNN, Error analysis, Elliptic equations, Rademacher complexity}

\begin{abstract}
We present a rigorous theoretical analysis of the convergence rate of the deep mixed residual method (MIM) when applied to a linear elliptic equation with various types of boundary conditions. The MIM method has been proposed as a more effective numerical approximation method compared to the deep Galerkin method (DGM) and deep Ritz method (DRM) in various cases. Our analysis shows that MIM outperforms DRM and deep Galerkin method for weak solution (DGMW) in the Dirichlet case due to its ability to enforce the boundary condition. However, for the Neumann and Robin cases, MIM demonstrates similar performance to the other methods. 
Our results provides valuable insights into the strengths of MIM and its comparative performance in solving linear elliptic equations with different boundary conditions. 
%MIM reformulates high-order PDEs into a first-order system and uses the residual of the first-order system in the least-squares sense as the loss function, enabling the exact boundary conditions to be enforced.
%The study splits the total error into approximation error, statistical error, and optimization error and compares its performance to existing results for DRM and DGMW. 
%We also compare the performance of MIM to existing results for DRM and deep Galerkin method for weak solution (DGMW) and find that MIM provides better numerical approximation for Dirichlet case, but shows similar results for  Neumann and Robin case.
\end{abstract}
\subjclass[2020]{65N12, 65N15.}
\maketitle

\section{Introduction}

Partial differential equations (PDEs) are widely used in various fields such as science, engineering, and finance \cite{beck2019machine,hutzenthaler2020overcoming} to model complex physical phenomena. Traditional numerical methods such as the finite difference method (FDM) and finite element method (FEM) have been successful in solving low-dimensional PDEs. However, when it comes to high-dimensional problems, the curse of dimensionality becomes a bottleneck, making classical methods less effective. In recent years, deep learning-based methods \cite{han2017deep} such as the deep Ritz method (DRM)
\cite{yu2018deep,han2018solving,chen2019quasi,duan2021analysis}, the deep Galerkin method (DGM) \cite{sirignano2018dgm}, the physics informed neural network method \cite{raissi2019physics}, the weak adversarial network method \cite{zang2020weak}, the deep least-squares methods \cite{CAI2020109707}, least-squares ReLU neural network (LSNN) method \cite{CAI2021110514} and the local deep learning method (LDPM) \cite{yang2021local} have shown significant potential for solving high-dimensional PDEs, offering an attractive alternative to traditional methods.

Despite the success of deep learning methods such as DGM and DRM, they still suffer from convergence efficiency issues, especially when dealing with boundary conditions. Recently, the (deep) mixed residual method (MIM) \cite{lyu2022mim} has been proposed to address these limitations. MIM is a novel approach that reformulates high-order PDEs into a first-order system. This idea is similar to the local discontinuous Galerkin method (LDGM) and the mixed finite element method (FEM). MIM also uses the residual of the first-order system in the least-squares sense as the loss function, which is closely related to the least-squares finite element method. Importantly, MIM can enforce exact boundary conditions, including mixed boundary conditions, which are difficult to deal with in previous deep learning methods \cite{CSIAM-AM-2-748}. In comparison to DGM and DRM, MIM provides better numerical approximation in most tested cases. However, rigorous analysis of the convergence rate of the MIM is still lacking, which limits our understanding of its theoretical properties and practical performance.

 The analysis of MIM was carried out in \cite{li2022priori} for the first time. This study asserts that MIM can approximate high-order derivatives more accurately than DRM, as confirmed by numerical experiments.  Here, the total error is divided into the approximation error and the quadrature error, while in this paper, it is divided into three parts. Additionally, they only talk about two-layer network case, but the depth and width of the network in our work are not limited. One of the highlights of \cite{li2022priori} is that the authors demonstrate that increasing the number of training samples is independent of network size.

In this article, we provide a rigorous theoretical analysis of the convergence rate of the MIM method applied to a linear elliptic equation with different kinds of boundary conditions, and compare its performance to existing results for DRM \cite{jiao2021error} and Deep Galerkin Method for weak solutions (DGMW) \cite{jiao_convergence_2023}. We use the same method that has been successfully applied to DRM and DGMW. To be more clear, \cite{duan2021analysis} presents a convergence rate in $H^1$ norm for DRM for Laplace equations with Dirichlet boundary condition by decomposing the total error into three parts, which are approximation error, statistical error, and the error caused by the boundary penalty. Specifically, we split the total error into approximation error, statistical error, and optimization error. The approximation error can be obtained by a recent work \cite{guhring2021approximation} and provides the information of the network that should be used. The derivation of statistical error is the main contribution of this paper, which uses Rademacher complexity to provide us the number of sampling points needed. The optimization error contains information about the iteration times and is beyond the scope of this paper. Although the main idea is similar to previous work \cite{jiao2021error}, there are some differences worth mentioning: (i) MIM uses the residual of the first-order system in the least-squares sense as the loss function, which needs an additional insensible variation condition on the coefficients of the elliptic problem to make the loss function non-degenerate. (ii) The neural network that MIM uses has a multi-dimensional output function, which makes the approximation analysis slightly different. (iii) An extra modification layer is added to the neural network to enforce exact boundary conditions, which complicates the estimates of statistical error.

\subsection{Results and Comparisons}

Let us now present our main results. Suppose that $\Omega\in \mathbb{R}^d$ is a bounded open set. Let $u_e\in H^3(\Omega)$ be the exact solution of the elliptic problem \eqref{elliptic problem}-\eqref{boundary condition}, and let $u_{\mathcal{A}}$ be the solution obtained by a random solver for MIM. Denote $\mathcal{N}{\rho}(\mathcal{D}, \mathfrak{n}_{\mathcal{D}}, B_{\theta})$ as the collection of functions implemented by a $\rho$-neural network with depth $\mathcal{D}$, a total number of nonzero weights $\mathfrak{n}_{\mathcal{D}}$, and all weights having a uniform bound of $B_{\theta}$. Setting $\rho=\frac{1}{1+e^{-x}}$ or $\frac{e^x-e^{-x}}{e^x+e^{-x}}$ as the activation function, we make the following claim:

For any $\epsilon > 0$ and $\mu \in (0,1)$, consider the parameterized function class $\mathcal{P}=\mathcal{N}_{\rho} (\mathcal{D},\mathfrak{n}_D,B_{\theta})$ with
    \begin{equation*}
    \mathcal{D} = C\log(d+1), \quad 
    \mathfrak{n}_D = C(d)\epsilon^{-d/(1-\mu)},\quad 
    B_{\theta} = C(d)\epsilon^{-(9d+8)/(2-2\mu)},
\end{equation*}
and the number of samples
\begin{equation*}\label{sampleing number}
 N = C(d,coe)
    \epsilon^{-4-\mathcal{D} [22d + 16]/(1-\mu)},
\end{equation*}
where $coe$ denotes the coefficients of the elliptic problem. If the optimization error is $\mathcal{E}_{opt} \leq \epsilon$, then
\begin{equation}
	\mathbb{E}_{\{\bX_i\}_{i=1}^N} ||u_{\mathcal{A}}-u_e||_{H^1(\Omega)} \leq C(coe)
	\epsilon .\nonumber
\end{equation}
%Here the constant $C(coe)$ depends only on the coefficients of the elliptic problem.
We have obtained a convergence rate of nearly $\mathcal{O}(N^{-\frac{1}{d}})$ for the error $\epsilon$, which is comparable to the work of \cite{jiao2021error,jiao_convergence_2023}. In particular, we provide a detailed comparison of our results with those of other works in Table \ref{tab}.

\begin{table}[]
    \centering
    \begin{tabular}{c|c|c|c}
    \toprule[1.5pt]
         Method &  \makecell[c]{Boundary \\Condition } 
         %& Depth $\mathcal{D} =$ 
         & Scaling $\epsilon = $ & \makecell[c]{Regularity \\Condition }\\ 
         \hline \rule{0pt}{16pt}
         \multirow{2}{*}{DRM\cite{jiao2021error}, DGMW\cite{jiao_convergence_2023}} 
         & Neumann, Robin 
         %& \multirow{3}{*}{$ C\log (d+1)$}  
         & $\mathcal{O}\big( N^{-\frac{1}{4+\mathcal{D}(22d + 16)(1-\mu)^{-1}}}\big)$
         & \multirow{2}{*}{$u_e \in H^2(\Omega)$}  \\
         \cline{2-3} \rule{0pt}{16pt}
          & Dilichlet 
          & $\mathcal{O}\big( N^{-\frac{1}{4+\mathcal{D}(55d + 40)(1-\mu)^{-1}}}\big)$ &  \\
           \cline{1-4} \rule{0pt}{16pt}
          MIM (this paper) 
          & \makecell[c]{Neumann, Robin, \\Dilichlet}  
          & $\mathcal{O}\big( N^{-\frac{1}{4+\mathcal{D}(22d + 16)(1-\mu)^{-1}}}\big)$ 
          & $u_e \in H^3(\Omega)$ \\
          \bottomrule[1.5pt]
    \end{tabular}
    \caption{ Previous works and our result. Here the depth of the network $\mathcal{D} = C\log (d+1)$, and $\mu\in(0,1)$.}
    \label{tab}
\end{table}

Based on the comparison in Table \ref{tab}, it can be concluded that MIM has an advantage over DRM and DGMW for the Dirichlet case due to its ability to enforce the boundary condition. However, for the Neumann and Robin cases, MIM shows similar results to the other methods. While MIM has shown better numerical approximation in most of the experiments conducted, it does not seem to be observable in our analysis. This could be attributed to the lack of analysis of the optimization error, which contains information about the iteration times. Additionally, it should be noted that MIM requires more regularity for the exact solution, as the strong form of the loss function in the least-squares sense has been used.

The rest of the paper is organized as follows: In Section \ref{sec: network and MIM}, we provide some preliminaries about Neural Network and Mixed Residual Method. In Section \ref{sec: proof of main result}, we present the error decomposition and a sketch of the proof of the main results. In Section \ref{sec: derive sta error}, we provide a detailed analysis of the statistical error. Finally, we conclude and give a short discussion in Section \ref{sec: conclusion}.

\section{Neural Network and Mixed Residual Method}\label{sec: network and MIM}
\subsection{Neural Network with modified layer}
Let us first introduce some notations of neural networks that concerned in our later analysis. 

Let $\mathcal{D}\in \mathbb{N}^+$ denote the depth of a network. We can define the function $\mathbf{f}$ implemented by a neural network by
\begin{equation}\label{network}
\begin{aligned}
    &\mathbf{f}_0(\bx) = \bx, \\
    &\mathbf{f}_{\ell}(\bx) = \rho (A_{\ell}\mathbf{f}_{\ell-1}+\mathbf{b}_{\ell})\quad \mbox{for $\ell = 1,\dots,\mathcal{D}-1$}, \\
    &\mathbf{f} \triangleq \mathbf{f}_{\mathcal{D}}(\bx) = A_{\mathcal{D}}\mathbf{f}_{\mathcal{D}-1}+\mathbf{b}_{\mathcal{D}}, 
\end{aligned}
\end{equation}
where $A_{\ell}=\left(a_{ij}^{(\ell)}\right)\in \mathbb{R}^{n_{\ell}\times n_{\ell-1}}$, $\mathbf{b}_{\ell}=\left(b_i^{(\ell)}\right)\in\mathbb{R}^{n_{\ell}}$, 
$n_{\ell}$ is the width of a certain layer of the network,
and $\rho$ is the activation function. In this article, we consider 
\begin{equation}\label{activation function}
    \rho=\frac{1}{1+e^{-x}} \quad \mbox{or}\quad 
    \rho=\frac{e^x-e^{-x}}{e^x+e^{-x}}.
\end{equation}
%$\rho=\frac{1}{1+e^{-x}}$ or $\frac{e^x-e^{-x}}{e^x+e^{-x}}$. 
While the network \eqref{network} is widely applied to DGM and DRM method, here we would like to introduce an extra modified layer in the framework of MIM, which can be written as
\begin{align}\label{modified layer}
    \mathbf{f}^{\mathbf{mod}}(\mathbf{x}) 
    \triangleq 
    \phi_g(\bx, \mathbf{f}(\bx)) 
    = \phi_g(\bx, \mathbf{f}_\mathcal{D}(\bx)) , 
\end{align}
here $\phi_g: \mathbb{R}^d \times \mathbb{R}^{n_{\mathcal{D}}} \rightarrow \mathbb{R}^{n_{\mathcal{D}}}$ is $C^1$ function depending on $g$ and $\Omega$, and is constructed to modify the output function for the boundary condition. Furthermore, we may assume that $n_{\mathcal{D}} = d+1$, which means that the output function is a vector-valued function of $d+1$ dimension.

Additionally, we define $\theta = \{A_{\ell}, \mathbf{b}_{\ell}\}_{\ell=1}^{\mathcal{D}}$ as the weight parameters, and use $\mathfrak{n}_{i},\ i = 1,\dots,\mathcal{D}$ to denote the number of nonzero weights of the first $i$ layers. 
Generally, the function $\mathbf{f}$ generated in \eqref{network} is called a parameterized function with respect to $\theta$, and we use the notation $\mathcal{N}_{\rho}(\mathcal{D}, \mathfrak{n}_{\mathcal{D}}, B_{\theta})$ to represent the collection of such functions that implemented by a $\rho$-neural network with depth $\mathcal{D}$, total number of nonzero weights $\mathfrak{n}_{\mathcal{D}}$, and all weights having a shared bound $B_{\theta}$.

For the analysis purpose, we conclude the following proposition for the activation function:
\begin{proposition}\label{assumption: activation function}
For the activation function $\rho$ given in \eqref{activation function}, we have that $\rho\in C^1$ satisfying the boundedness condition:
\begin{equation*}
    |\rho(x)|\le 1, \qquad  
    |\rho'(x)|\le 1,
\end{equation*}
for any $x\in \mathbb{R}$, and the Lipschitz condition:
\begin{equation*}
    |\rho(x)-\rho(y)|\le |x-y|,
    \qquad 
    |\rho'(x)-\rho'(y)|\le |x-y|,
\end{equation*}
for any $x,y\in \mathbb{R}$.
\end{proposition}
We also make the following assumption on modified function:
\begin{assumption}\label{assumption: modified function}
For the modified function $\phi_g$ in \eqref{modified layer}, we assume $\phi_g\in C^1$ satisfying
the boundedness condition:
\begin{equation*}
\begin{gathered}
    |\phi_g(\bx,\mathbf{f}(\bx))|
    \le 
    |\mathbf{f}(\bx)| + B_{g}, \\
    |\partial_{x_p} \phi_g(\bx,\mathbf{f}(\bx))|
    \le 
    |\partial_{x_p}\mathbf{f}(\bx)| + B_{\phi'}|\mathbf{f}(\bx)| + B_{g'},
    \end{gathered}
\end{equation*}
for any $\bx\in \mathbb{R}^d$,
and the Lipschitz condition with respect to $\mathbf{f}$:
\begin{equation*}
\begin{gathered}
    |\phi_g(\bx,\mathbf{f}(\bx)) - \phi_g(\bx,\widetilde{\mathbf{f}}(\bx))| 
    \le 
    |\mathbf{f}(\bx) - \widetilde{\mathbf{f}}(\bx)|,\\
   |\partial_{x_p}\phi_g(\bx,\mathbf{f}(\bx)) 
   - 
   \partial_{x_p}\phi_g(\bx,\widetilde{\mathbf{f}}(\bx))| 
    \le 
    |\partial_{x_p}\mathbf{f}(\bx) - \partial_{x_p}\widetilde{\mathbf{f}}(\bx)|
    + B_{\phi'}
    |\mathbf{f}(\bx) - \widetilde{\mathbf{f}}(\bx)|,
\end{gathered}
\end{equation*}
for any $\bx\in \mathbb{R}^d$. Here $B_{g}$, $B_{g'}$ depend on the boundary condition $g$, $B_{\phi'}$ depends on the area $\Omega$ and construction of $\phi_g$.
\end{assumption}
In the following, we introduce the Mixed Residual Method and provide an example of constructing a modified function $\phi_g$ for the case where $\Omega$ is a ball. We demonstrate that Assumption \ref{assumption: modified function} holds in this case.

\subsection{Mixed Residual Method}
Let $\Omega$ be a convex bounded open set in $\mathbb{R}^d$, and $\partial\Omega\in C^{\infty}$. For simplicity, we further assume that $\Omega\subset [0,1]^d$. We consider the following second order elliptic equation:
\begin{equation}\label{elliptic problem}
    -\Delta u + \omega u = f \quad  \mbox{in $\Omega$},
\end{equation}
with three kinds of boundary conditions:
\begin{equation}\label{boundary condition}
\begin{aligned}
    u & =  g \quad \mbox{on $\partial\Omega$}, \\
    \frac{\partial u}{\partial \nu} 
    & = g \quad \mbox{on $\partial\Omega$}, \\
    \alpha u + \frac{\partial u}{\partial \nu} &= g \quad \mbox{on $\partial\Omega$}, \quad \alpha > 0,
\end{aligned}
\end{equation}
namely Dirichlet, Neumann, and Robin boundary condition, respectively. Here we make the following assumption for this problem: 
\begin{assumption}\label{assumption: regularity}
Consider the elliptic problem \eqref{elliptic problem}-\eqref{boundary condition}, we assume the inhomogeneous term $f\in L^{\infty}(\Omega)\cap H^1(\Omega)$, the boundary term $g\in C^{2}(\partial\Omega)$, and coefficient $\omega\in C^2(\Omega)$. Moreover, $\omega$ has positive lower bound
    %$\omega(\bx) \ge c_{\omega}>0$ for any $\bx\in \Omega$, 
    \begin{equation*}
        \omega(\bx) \ge c_{\omega}>0 \quad \mbox{for any $\bx\in \Omega$},
    \end{equation*}
   and a relatively insensible variation
    \begin{equation*}
        \|\nabla \omega\|_{L^{\infty}(\Omega)}
        < (c_{\omega})^{\frac{3}{2}}.
    \end{equation*}
\end{assumption}
Under the regularity condition and lower bound of $\omega$ in the Assumption \ref{assumption: regularity}, one can conclude that problem \eqref{elliptic problem}-\eqref{boundary condition} has a unique weak solution $u^*\in H^3(\Omega)$.
The relatively insensible variation of $\omega$ is used to make the loss function non-degenerate in the decomposition of error.
In this paper we use the notation $C(coe) =  C( f, g,\omega,\Omega)$ to denote a constant depending upon the given problem \eqref{elliptic problem}-\eqref{boundary condition}.

%Then we talk about Mixed Residual Method(MIM). 
Now from the MIM's point of view, we can reconsider problem \eqref{elliptic problem} as a first-order system. In fact, we introduce new variable $\bp = \nabla u$, and consider $(u,\bp)$ as independent $(d+1)$-dimension variable, which can be approximated by the output of neural network $\mathcal{P}$ defined in \eqref{network}-\eqref{modified layer}, of solving a minimization problem 
\begin{equation*}
    \min_{u,\bp \in \mathcal{P}} \mathcal{L}(u,\bp),
\end{equation*}
here the loss function is defined by
\begin{equation}\label{minimization problem}
\mathcal{L}(u,\bp) := \int_{\Omega} \omega \big(\bp-\nabla u \big)^2 + \big(-\bdiv \bp + \omega u - f\big)^2 \d\bx.
\end{equation}
Note that we have no penalty terms of boundary condition in the problem \eqref{minimization problem}, since the extra modification layer \eqref{modified layer} has been constructed such that the output function automatically fit the boundary condition. Here, we would like to mention that such a vanishing of penalty terms is the most highlight point in the networks of MIM. 

While the construction of such modification layer relies on the structure of area $\Omega$, we introduction a simple example here for the case that $\Omega$ is a ball. Suppose now $\Omega = \{ \bx\in \mathbb{R}^d:\, |\bx|^2\le 1\}$, denote $\phi_g(\bx,\mathbf{f}(\bx)) = \{\phi_{g,i}\}_{i=1}^{d+1}$ and $\mathbf{f} = \{f_i\}_{i=1}^{d+1}$ in \eqref{modified layer}. Then we can set the modification for Dirichlet case:
\begin{equation*}
\left\{\begin{aligned}
    \phi_{g,1}(\bx,\mathbf{f}(\bx)) & =   (1-|\bx|)^2 f_1(\bx) + \tilde{g}(\bx) |\bx|, \\
      \phi_{g,i}(\bx,\mathbf{f}(\bx)) & = f_i(\bx),\quad i=2,\dots,d+1,
\end{aligned}\right.
\end{equation*}
for Neumann case:
\begin{equation*}
\left\{\begin{aligned}
    \phi_{g,1}(\bx,\mathbf{f}(\bx)) & =  f_1, \\
      \phi_{g,i}(\bx,\mathbf{f}(\bx)) 
      & = 
      (1-|\bx|)^2 f_i(\bx) +  \tilde{g}(\bx)x_i,\quad i=2,\dots,d+1,
\end{aligned}\right.
\end{equation*}
and for Robin case:
\begin{equation*}
\left\{\begin{aligned}
    \phi_{g,1}(\bx,\mathbf{f}(\bx)) & =  f_1, \\
      \phi_{g,i}(\bx,\mathbf{f}(\bx)) 
      & = 
      (1-|\bx|)^2 f_i(\bx) + \big(\tilde{g}(\bx)-f_1\big) x_i,\quad i=2,\dots,d+1,
\end{aligned}\right.
\end{equation*}
here $\tilde{g}(\bx)\in C^3(\Omega)$ such that $\tilde{g}(\bx) = g(\bx)$ for any $\bx\in \partial\Omega$. Following the above construction, one can check the output of network \eqref{network}-\eqref{modified layer} satisfies the boundary condition given in \eqref{boundary condition} automatically by the notations of
\begin{equation*}
    (u,\bp )= \mathbf{f}^{\mathbf{mod}} = \phi_g(\bx, \mathbf{f}(\bx)).
\end{equation*}

\section{Proof of Main Result}\label{sec: proof of main result}
In this section we give a sketch of the poof of our main result. Note that in the statistical point of view, it is equivalent to write the loss function \eqref{minimization problem} in the form of
\begin{equation}
    \mathcal{L}(u,\bp) = |\Omega| \mathbb{E}_{\bX \sim U(\Omega)} 
    \Big[\omega(\bX)\big(\bp(\bX) - \nabla u(\bX)\big)^2 
    + \big(- \bdiv \bp(\bX) + \omega(\bX) u(\bX) - f(\bX)\big)^2\Big], \nonumber
\end{equation}
where $U(\Omega)$ is the uniform distribution on $\Omega$.
As one can sample $N$ points to obtain an approximation of $L(u,\bp)$, we define its discrete version as
\begin{equation}
    \widehat{\mathcal{L}}^N(u,\bp) 
    =
     \frac{|\Omega|}{N} \sum_{k=1}^N 
    \Big[\omega(\bX_k)\big(\bp(\bX_k) - \nabla u(\bX_k)\big)^2 + \big(- \bdiv \bp(\bX_k) + \omega(\bX_k) u(\bX_k) - f(\bX_k)\big)^2\Big]. \nonumber
\end{equation}
where $\{\bX_k\}_{k=1}^N$ are i.i.d. random variables according to $U(\Omega)$.

\subsection{Error Decomposition}
Now let us set $(u_{e},\bp_{e})$ to be the exact solution of problem \eqref{elliptic problem}, thus is also the minimizer of Loss Function $\mathcal{L}(u,\bp)$. We also consider the problem 
\begin{equation*}
	\min_{u,\bp \in \mathcal{P}} \widehat{\mathcal{L}}^N(u,\bp),
\end{equation*}
and denote its solution by $(\hat{u}_{\theta},\hat{\bp}_{\theta})$. Finally, let us represent an SGD algorithm by $\mathcal{A}$, and denote its output solution by $(u_{\mathcal{A}}, \bp_{\mathcal{A}})$. Now we introduce the approximation error which denotes the difference between $(u_{e},\bp_{e})$ and its projection onto $\mathcal{P}$: 
\begin{equation*}
    \mathcal{E}_{app} = \inf_{u,\bp \in \mathcal{P}}
    \big(\|u-u_e\|_{H^1(\Omega)}^2 + \|\bp-\bp_e\|_{H^1(\Omega)}^2\big),
\end{equation*}
the statistical error between $\mathcal{L}$ and $\widehat{\mathcal{L}}^N$:
\begin{equation*}
    \mathcal{E}_{sta} = \sup_{u,\bp \in \mathcal{P}} \big[\mathcal{L}(u,\bp)-\widehat{\mathcal{L}}^N(u,\bp)\big] + \sup_{u,\bp \in \mathcal{P}}\big[\widehat{\mathcal{L}}^N(u,\bp) - \mathcal{L}(u,\bp)\big],
\end{equation*}
and the optimization error between $(\hat{u}_{\theta},\hat{\bp}_{\theta})$ and $(u_{\mathcal{A}}, \bp_{\mathcal{A}})$:
\begin{equation*}
    \mathcal{E}_{opt} = \widehat{\mathcal{L}}^N(u_{\mathcal{A}},\bp_{\mathcal{A}}) - \widehat{\mathcal{L}}^N(\hat{u}_{\theta},\hat{\bp}_{\theta}).
\end{equation*}
Then we have the following decomposition proposition:

\begin{proposition}
Under the Assumption \ref{assumption: regularity}, it holds that the total error between exact solution $(u_{e},\bp_{e})$ and output $(u_{\mathcal{A}}, \bp_{\mathcal{A}})$ can be decomposed as:
     %Assume that $\mathcal{P}\subset H^1(\Omega)$. Let $(u_{e},\bp_{e})$ be the exact solution of problem \eqref{elliptic problem}. Let $u_{\mathcal{A}}$ be the solution of problem $\min_{u,\bp \in \mathcal{P}} \widehat{\mathcal{L}}(u,\bp)$ generated by a random solver.
    \begin{equation}\label{error decomposition}
        \|u_{\mathcal{A}} - u_e\|_{H^1(\Omega)} +  \|\bp_{\mathcal{A}} - \bp_e\|_{L^2(\Omega)}
        \leq C(coe)\big[\mathcal{E}_{app}+\mathcal{E}_{sta}+\mathcal{E}_{opt}\big]^{1/2}.
    \end{equation}
\end{proposition}
\begin{proof}
For any $(u, \bp)\in\mathcal{P}$, set $\bar{u}=u-u_e$, $\bar{\bp}=\bp-\bp_e$, then we have
\begin{equation}
\begin{aligned}  \nonumber
     \mathcal{L}(u,\bp) 
     = &
     \mathcal{L}(\bar{u} + u_e, \bar{\bp} + \bp_e) \\
     = &
     \mathcal{L}(u_e,\bp_e) 
     + \int_{\Omega}|\bdiv \bar{\bp} - \omega\bar{u} |^2 
     + \omega |\nabla \bar{u} - \bar{\bp} |^2 \d \bx\\
     & +
     2\int_{\Omega} \big( \bdiv \bp_e - \omega u_e + f\big) \big( \bdiv \bar{\bp} - \omega \bar{u} \big) 
     + \omega \big( \nabla u_e - \bp_e\big) \big( \nabla \bar{u} - \bar{\bp} \big) \d \bx \\
     =& 
     \mathcal{L}(u_e,\bp_e) 
     + \int_{\Omega}|\bdiv \bar{\bp} - \omega\bar{u} |^2 
     + \omega |\nabla \bar{u} - \bar{\bp} |^2 \d \bx.
\end{aligned}
\end{equation}
The last line holds due to the fact that $(u_R,\bp_R)$ is the real solution of the equation.
Thus we can conclude that
\begin{equation}\label{upper bound of difference of L}
    \begin{aligned}
        \mathcal{L}(u,\bp)-\mathcal{L}(u_e,\bp_e)
        =&
        \int_{\Omega}|\bdiv \bar{\bp} - \omega\bar{u} |^2 
     + \omega |\nabla \bar{u} - \bar{\bp} |^2 \d \bx\\
        &\leq 
        2 
        \big( \|\omega\|_{L^{\infty}(\Omega)}^2 \|\bar{u}\|_{H^1(\Omega)}^2 
        + \|\bar{\bp}\|_{H^1(\Omega)}^2 \big).
    \end{aligned}
\end{equation}
On the other hand, by using integration by parts one has
\begin{equation}\label{lower bound}
    \begin{aligned}
        \int_{\Omega}|\bdiv \bar{\bp} - \omega\bar{u} |^2 
     + \omega |\nabla \bar{u} - \bar{\bp} |^2 \d \bx
     = &
     \int_{\Omega}|\bdiv \bar{\bp} |^2  
     + | \omega\bar{u} |^2 
     + \omega |\nabla \bar{u} |^2 + \omega | \bar{\bp} |^2 \d \bx\\
     & - 2
      \int_{\partial\Omega}\omega\bar{u} (\boldsymbol{\nu} \cdot \bar{\bp}) \d \bx 
      + 2\int_{\Omega} \bar{u} (\nabla \omega \cdot \bar{\bp})\d \bx .
    \end{aligned}
\end{equation}
Here the boundary term equals $0$ for the Dirichlet and Neumann condition. As for the Robin case, we have that
\begin{equation*}
    \begin{aligned}
      - 2 \int_{\partial\Omega}\omega\bar{u} (\boldsymbol{\nu} \cdot \bar{\bp}) \d \bx 
      =
      2\alpha \int_{\partial\Omega}\omega\bar{u}^2 \d \bx \ge 0.
    \end{aligned}
\end{equation*}
As for the last term in \eqref{lower bound}, we have the estimate
\begin{equation*}
    \begin{aligned}
      2\int_{\Omega} \bar{u} (\nabla \omega \cdot \bar{\bp})\d \bx 
      \ge -
      \|\nabla \omega\|_{L^{\infty}(\Omega)}^{\frac{4}{3}} \|\bar{u}\|_{L^2(\Omega)}^2
      -
      \|\nabla \omega\|_{L^{\infty}(\Omega)}^{\frac{2}{3}}
        \|\bar{\bp}\|_{L^2(\Omega)}^2.
    \end{aligned}
\end{equation*}
Using above estimates and turn back to \eqref{lower bound}, we finally obtain
\begin{equation}\label{lower bound final result}
    \begin{aligned}
        \int_{\Omega}|\bdiv \bar{\bp} - \omega\bar{u} |^2 
     +& \omega |\nabla \bar{u} - \bar{\bp} |^2 \d \bx
     \ge 
     c_{\omega} \|\nabla\bar{u}\|_{L^2(\Omega)}^2\\
     & + 
     (c_{\omega}^2 - \|\nabla \omega\|_{L^{\infty}(\Omega)}^{\frac{4}{3}} ) \|\bar{u}\|_{L^2(\Omega)}^2
     + 
     (c_{\omega} - \|\nabla \omega\|_{L^{\infty}(\Omega)}^{\frac{2}{3}})
        \|\bar{\bp}\|_{L^2(\Omega)}^2.
    \end{aligned}
\end{equation}
Combining \eqref{upper bound of difference of L} and \eqref{lower bound final result}, and using Assumption \ref{assumption: regularity} of $\omega$, it leads to
\begin{equation*}
    \begin{aligned}
        C \big( \|u-u_e\|_{H^1(\Omega)}^2 
        +& \|\bp-\bp_e\|_{L^2(\Omega)}^2 \big)\\
        \le &
        L(u,\bp)-L(u_e,\bp_e)
        \leq 
        C
        \big( \|u-u_e\|_{H^1(\Omega)}^2 
        + \|\bp-\bp_e\|_{H^1(\Omega)}^2 \big).
    \end{aligned}
\end{equation*}

Now, let us divide the difference of Loss function as
\begin{equation}\nonumber
    \begin{aligned}
        &\mathcal{L}(u_{\mathcal{A}},\bp_{\mathcal{A}})-\mathcal{L}(u_e,\bp_e)\\
        =& 
        \mathcal{L}(u_{\mathcal{A}},\bp_{\mathcal{A}})
        -
        \widehat{\mathcal{L}}^N(u_{\mathcal{A}},\bp_{\mathcal{A}})
        +
        \widehat{\mathcal{L}}^N(u_{\mathcal{A}},\bp_{\mathcal{A}})
        -
        \widehat{\mathcal{L}}^N(\hat{u}_{\theta},\hat{\bp}_{\theta})+\widehat{L}(\hat{u}_{\theta},\hat{\bp}_{\theta})\\
        &-
        \widehat{\mathcal{L}}^N(u,\bp)
        +
        \widehat{\mathcal{L}}^N(u,\bp)
        -
        \mathcal{L}(u,\bp)+\mathcal{L}(u,\bp)-L(u_e,\bp_e)\\
        \leq & 
        \sup_{u,\bp\in\mathcal{P}}\big[\mathcal{L}(u,\bp)-\widehat{\mathcal{L}}^N(u,\bp)\big]
        +
        \big[\widehat{\mathcal{L}}(u_{\mathcal{A}},\bp_{\mathcal{A}})
        -
        \widehat{\mathcal{L}}^N(\hat{u}_{\theta},\hat{\bp}_{\theta})\big]\\
        & +
        \sup_{u,\bp\in \mathcal{P}}\big[\widehat{\mathcal{L}}^N(u,\bp)
        -
        \mathcal{L}(u,\bp)\big]
        +
        \|u-u_e\|_{H^1(\Omega)}^2 + \|\bp-\bp_e\|_{H^1(\Omega)}^2.
    \end{aligned}
\end{equation}  
Since $(u,\bp)$ can be any element in $\mathcal{P}$, we take the infimum of $(u,\bp)\in \mathcal{P}$ on right-hand sides of above estimate, and derive
\begin{equation}\nonumber
    \begin{aligned}
        \mathcal{L}(u_{\mathcal{A}},\bp_{\mathcal{A}})
        - \mathcal{L}(u_e,\bp_e)
        \leq
        \sup_{u,\bp\in\mathcal{P}}\big[\mathcal{L}(u,\bp)-\widehat{\mathcal{L}}^N(u,\bp)\big]
        +
        \sup_{u,\bp\in \mathcal{P}}\big[\widehat{\mathcal{L}}^N(u,\bp)-\mathcal{L}(u,\bp)\big]\\
        +
        \big[\widehat{\mathcal{L}}^N(u_{\mathcal{A}},\bp_{\mathcal{A}})
        -
        \widehat{\mathcal{L}}^N(\hat{u}_{\theta},\hat{\bp}_{\theta})\big]
        + 
        \inf_{u,\bp\in \mathcal{P}} \big(\|u-u_e\|_{H^1(\Omega)}^2 + \|\bp-\bp_e\|_{H^1(\Omega)}^2\big).
    \end{aligned}
\end{equation}
Combining the discussion above yields \eqref{error decomposition}.

\end{proof}

\subsection{Approximation Error}
Now we consider the approximation of the neural network in Sobolev spaces, which is comprehensively studied in \cite{guhring2021approximation} for a variety of activation functions. The key idea in \cite{guhring2021approximation} to build the upper bound in Sobolev spaces is to construct an approximate partition of unity.
Denote $\mathcal{F}_{s,p,d}\triangleq\{f\in W^{s,p}\left([0,1]^d\right): \|f\|_{W^{s,p}\left([0,1]^d\right)}\leq 1\}$. Let $\rho$ be logistic function $\frac{1}{1+e^{-x}}$ or tanh function $\frac{e^x-e^{-x}}{e^x+e^{-x}}$. We introduce the result from  Proposition 4.8, \cite{guhring2021approximation}:

\begin{theorem}
    Let $p\geq 1$, $s, k, d\in\mathbb{N}^+$, $s\geq k+1$.  For any $\epsilon >0$ and $f\in \mathcal{F}_{s,p,d}$, there exists a neural network $f_{\rho}$ with depth $C\log(d+s)$ and $C(d,s,p,k)\epsilon^{-d/(s-k-\mu k)}$ non-zero weights such that
		\begin{equation}
			\|f-f_{\rho}\|_{W^{k,p}\left([0,1]^d\right)}\leq \epsilon. \nonumber
		\end{equation}
		Moreover, the weights in the neural network are bounded in absolute value by
		\begin{equation}
			C(d,s,p,k)\epsilon^{-2-\frac{2(d/p+d+k+\mu k)+d/p+d}{s-k-\mu k}} \nonumber
		\end{equation}
		where $\mu$ is an arbitrarily small positive number.
\end{theorem}

Note that the domain $\Omega$ is a subset of the unit cube $[0,1]^d$. The following extension result can be used.

\begin{lemma}
    Let $k\in\mathbb{N}^+$, $1\leq p<\infty$. There exists a linear operator $E$ from $W^{k,p}(\Omega)$ to $W_0^{k,p}\left([0,1]^d\right)$ and $Eu=u$ in $\Omega$.
\end{lemma}

From above results, we can conclude by choosing $p=2$, $k=1$ and $s=2$ that
\begin{lemma}\label{lemma: approximation}
    Suppose $\mathbf{f}\in H^2(\Omega)$ is a vector valued function of $(d+1)$ dimension and satisfies $\|\mathbf{f}\|_{H^2(\Omega)}\leq 1$. For any $\epsilon > 0$,
    set the parameterized function class $\mathcal{P}=\mathcal{N}_{\rho} (\mathcal{D},\mathfrak{n}_D,B_{\theta})$ of network \eqref{network}-\eqref{modified layer} as
    \begin{equation*}
    \mathcal{D} = C\log(d+1), \quad 
    \mathfrak{n}_D = C(d)\epsilon^{-d/(1-\mu)},\quad 
    B_{\theta} = C(d)\epsilon^{-(9d+8)/(2-2\mu)},
\end{equation*}
with $\mu \in (0,1)$, then there exists a neural network $\mathbf{f}_{\rho}\in \mathcal{P}$, 
%with depth $C\log(d+1)$ and $C(d)\epsilon^{-d/(1-\mu)}$ non-zero weights, 
such that
		\begin{equation}
			\|\mathbf{f}-\mathbf{f}_{\rho}\|_{H^1(\Omega)}\leq \epsilon. \nonumber
		\end{equation}
		%Moreover, the weights in the neural network are bounded by $C(d)\epsilon^{-\frac{9d+8}{2-2\mu}}$, where $\mu\in(0,1)$ is an arbitrarily small positive number.
\end{lemma}

\subsection{Statistical Error and Convergence Rate}
For the statistical error, We have the following conclusion:
\begin{theorem}\label{thm: Statistical Error}
    Let the parameterized function class $\mathcal{P}=\mathcal{N}_{\rho} (\mathcal{D},\mathfrak{n}_D,B_{\theta})$ of network \eqref{network}-\eqref{modified layer} satisfying Assumption \ref{assumption: modified function}, then we have
		\begin{equation*}
        \mathbb{E}_{\left\{\bX_i\right\}_{i=1}^N} \sup _{u,\bp \in \mathcal{P}} \pm \big[\mathcal{L}(u,\bp) - \widehat{\mathcal{L}}^N(u,\bp)\big] 
        \leq 
        C(\operatorname{coe}) \frac{d \sqrt{\mathcal{D}} \mathfrak{n}_{\mathcal{D}}^{2 \mathcal{D}} B_\theta^{2 \mathcal{D}}}{\sqrt{N}} \sqrt{\log \left(d \mathcal{D} \mathfrak{n}_{\mathcal{D}} B_\theta N\right)}.
    \end{equation*}
\end{theorem}
The proof of Theorem \ref{thm: Statistical Error} is the main contribution of this paper, which will be presented in detail in Section \ref{sec: derive sta error}. Combining Lemma \ref{lemma: approximation} and Theorem \ref{thm: Statistical Error}, we can obtain our main result on the convergence rate:

%\subsection{Convergence Rate for the Mixed Residual Method}
\begin{theorem}
Given any $\epsilon > 0$,
    set the parameterized function class $\mathcal{P}=\mathcal{N}_{\rho} (\mathcal{D},\mathfrak{n}_D,B_{\theta})$ of network \eqref{network}-\eqref{modified layer} as
    \begin{equation}\label{parameterized function class}
    \mathcal{D} = C\log(d+1), \quad 
    \mathfrak{n}_D = C(d)\epsilon^{-d/(1-\mu)},\quad 
    B_{\theta} = C(d)\epsilon^{-(9d+8)/(2-2\mu)},
\end{equation}
with $\mu \in (0,1)$, and the number of samples
\begin{equation}\label{sampleing number}
    N = C(d,coe)
    \epsilon^{-4-\mathcal{D} [22d + 16]/(1-\mu)}.
\end{equation}
If Assumption \ref{assumption: modified function}-\ref{assumption: regularity} hold and the optimization error $\mathcal{E}_{opt} \leq \epsilon$, then
\begin{equation}
	\mathbb{E}_{\{\bX_i\}_{i=1}^N} ||u_{\mathcal{A}}-u_e||_{H^1(\Omega)} \leq C(coe)
	\epsilon .\nonumber
\end{equation}
\end{theorem}
\begin{proof}
    First let us normalize the solution:
    \begin{equation}\label{normalization}
        \begin{aligned}
            &\inf_{u,\bp \in\mathcal{P}}
            \Big\{\|u - u_{e}\|_{H^1(\Omega)} 
            + \|\bp - \bp_e\|_{L^2(\Omega)}\Big\}\\
            \le & 
            C \inf_{u,\bp \in\mathcal{P}}
            \Bigg\{\Big\|\frac{u}{\|u_e\|_{H^1(\Omega)}} - \frac{u_e}{\|u_e\|_{H^1(\Omega)}} \Big\|_{H^1(\Omega)} 
            +
            \Big\|\frac{\bp}{\|\bp_e\|_{H^1(\Omega)}} - \frac{\bp_e}{\|\bp_e\|_{H^1(\Omega)}} \Big\|_{L^2(\Omega)} 
            \Bigg\}\\
            =& 
            C \inf_{u,\bp \in\mathcal{P}}
            \Bigg\{\Big\|u - \frac{u_e}{\|u_e\|_{H^1(\Omega)}} \Big\|_{H^1(\Omega)} 
            +
            \Big\|\bp - \frac{\bp_e}{\|\bp_e\|_{H^1(\Omega)}} \Big\|_{L^2(\Omega)} 
            \Bigg\}.
        \end{aligned}
    \end{equation}
    Let the parameterized function class $\mathcal{P}$ be given in \eqref{parameterized function class}. It follows from the discussion in Section \ref{sec: network and MIM} that $u_e,\bp_e \in H^2(\Omega)$, thus one can apply Lemma \ref{lemma: approximation} to deduce there exists a neural network function $(u^*,\bp^*)\in\mathcal{P}$,
    such that
    \begin{equation}
        \Big\|u^* - \frac{u_e}{\|u_e\|_{H^1(\Omega)}} \Big\|_{H^1(\Omega)} 
            +
            \Big\|\bp^* - \frac{\bp_e}{\|\bp_e\|_{H^1(\Omega)}} \Big\|_{L^2(\Omega)}
            \leq \epsilon.\nonumber
    \end{equation}
    Therefore it follows from \eqref{normalization} that,
    \begin{equation}
        \mathcal{E}_{app}
        =
        \inf_{u,\bp \in \mathcal{P}}
    \big(\|u-u_e\|_{H^1(\Omega)}^2 + \|\bp-\bp_e\|_{H^1(\Omega)}^2\big)
    \le C\epsilon^2.
    \end{equation}
    On the other hand, by apply Theorem \ref{thm: Statistical Error} with $\mathcal{D}$, $\mathfrak{n}_{\mathcal{D}}$, $B_{\theta}$ and $N$ given in \eqref{parameterized function class}-\eqref{sampleing number}, one can obtain $\mathcal{E}_{sta}\leq C(d,coe)\epsilon^2$.
    %$D=C\mathrm{log}(d+1),\mathfrak{n}_{\mathcal{D}}=C(d)\left(\frac{1}{\beta^{3/2}\epsilon}\right)^{d/(1-\mu)}, B_{\theta}=C(d)\left(\frac{1}{\beta^{3/2}\epsilon}\right)^{(9d+8)/(2-2\mu)}$. 
    Thus the Theorem is proved.
\end{proof}

\section{Derivation of Statistical Error}\label{sec: derive sta error}
In this section, we investigate statistical error $\mathcal{E}_{sta}$, which is the paramount part in this article. 
Revisit the definition
\begin{equation*}
    \mathcal{E}_{sta} = \sup \limits_{u,\bp \in \mathcal{P}}  \pm \left [ \mathcal{L}(u,\bp) - \widehat{\mathcal{L}}^N(u,\bp) \right ].
\end{equation*}
For the sake of simplicity, let us denote the density of loss function by $\Gamma_k$, $k=1,\dots,8$, which are given by
\begin{equation*}
    \begin{aligned}  
    &\Gamma_1(u,\bp, \bx) 
    = |\bp(\bx)|^2, 
    &\Gamma_2(u,\bp,\bx) 
    =& -2  \bp(\bx) \cdot \nabla u(\bx),   \\   
    &\Gamma_3(u,\bp,\bx) 
    =  |\nabla u(\bx)|^2, 
    &\Gamma_4(u,\bp,\bx) 
    =& (\divp(\bx))^2,   \\ 
    &\Gamma_5(u,\bp,\bx) 
    = - 2\divp(\bx)u(\bx)\omega(\bx), 
    &\Gamma_6(u,\bp,\bx) 
    =& 2  \divp(\bx)f(\bx), \\ 
    &\Gamma_7(u,\bp,\bx) 
    =  \omega^2(\bx) u^2(\bx), 
    &\Gamma_8(u,\bp,\bx) 
    =& -2 u(\bx) \omega(\bx) f(\bx). 
\end{aligned}
\end{equation*}
And denote the expectation and $N$-particle average of $\Gamma_k$ by $\mathcal{L}_k$ and $\widehat{\mathcal{L}}_k^N$ for $k=1,\dots,8$, separately, which are defined by
\begin{equation}\label{define Lk}
    \mathcal{L}_k(u,\bp) 
    = |\Omega| E_{\bX\sim U(\Omega)} \Gamma_k(u,\bp, \bX),
    \qquad
    \widehat{\mathcal{L}}_k^N(u,p)
    =
    \frac{|\Omega|}{N} \sum_{i=1}^N \Gamma_k(u,\bp, \bX_i),
\end{equation}
here $\{\bX_i\}_{i=1}^N$ are N i.i.d random variable yielding the uniform distribution on $\Omega$.
Then it follows directly from the triangle inequality that
\begin{equation}\label{split of statistical error}
\begin{aligned}
    E_{\{\bX_i\}_{i=1}^N} \sup_{u,\bp \in \mathcal{P}}  \pm& \left [ \mathcal{L}(u,\bp) - \widehat{\mathcal{L}}^N(u,\bp) \right ] \\
    \leq& \sum_{k=1}^8 E_{\{\bX_i\}_{i=1}^N} \sup_{u,\bp \in \mathcal{P}}  \pm \left [ \mathcal{L}_k(u,\bp) - \widehat{\mathcal{L}}_k^N(u,\bp) \right ].
    \end{aligned}
\end{equation}
%where
%\begin{equation}
%\begin{aligned}   \nonumber
 %   &\mathcal{L}_1(u,p) 
 %   = |\Omega| E_{X\sim U(\Omega)} |p(X)|^2, 
 %   &\mathcal{L}_2(u,p) 
 %   =& -2 |\Omega| E_{X\sim U(\Omega)} |p(X)| \cdot |\nabla u(x)|,   \\   \nonumber
%    &\mathcal{L}_3(u,p) 
%    = |\Omega| E_{X\sim U(\Omega)} |\nabla u(X)|^2, 
%    &\mathcal{L}_4(u,p) 
%    =& |\Omega| E_{X\sim U(\Omega)} (\nabla \cdot p(X))^2,   \\ \nonumber
%    &\mathcal{L}_5(u,p) 
%    = -2 |\Omega| E_{X\sim U(\Omega)} \nabla \cdot p(X)u(X)\omega(X), &\mathcal{L}_6(u,p) 
%    =& 2 |\Omega| E_{X\sim U(\Omega)} \nabla \cdot p(X)f(X), \\ \nonumber
%    &\mathcal{L}_7(u,p) 
%    = |\Omega| E_{X\sim U(\Omega)} \omega^2(X) u^2(X), 
%    &\mathcal{L}_8(u,p) 
%    =& -2 |\Omega| E_{X\sim U(\Omega)} u(X) \omega(X) f(X). \nonumber
%\end{aligned}
%\end{equation}

%Proof. Direct result from trangle inequality.
\subsection{Rademacher complexity}
We aim to bound the difference between continuous loss $\mathcal{L}_k$ and empirical loss $\widehat{\mathcal{L}}_k^N$ by Rademacher complexity. First, let's introduce the definition of Rademacher complexity.

%The Rademacher complexity of a set $A \subseteq \mathbb{R}^N$ is defined as 
%\begin{equation}
%    \mathfrak{R}_N(A) = \mathbb{E}_{\{\sigma_k\}_{k=1}^N} \left[ \sup_{a \in A} \frac{1}{N} \sum_{k=1}^N \sigma_k a_k \right],
%\end{equation}
%where ${\sigma_k}_{k=1}^N$ are N $i.i.d$ Rademacher variable with %$\mathbb{P}(\sigma_k=1) = \mathbb{P}(\sigma_k=-1)=\frac{1}{2}$. 
\begin{definition}
The Rademacher complexity of function class $\mathcal{F}$ associate with random sample $\{\bX_i\}_{i=1}^N$ is defined as 
\begin{equation}
    \mathfrak{R}_N(\mathcal{F}) = \mathbb{E}_{\{\bX_i,\sigma_i\}_{k=1}^N} \left[ \sup_{u \in \mathcal{F}} \frac{1}{N} \sum_{i=1}^N \sigma_i u(\bX_i) \right],
\end{equation}
where $\{\sigma_i\}_{i=1}^N$ are N i.i.d Rademacher variable with $\mathbb{P}(\sigma_k=1) = \mathbb{P}(\sigma_k=-1)=\frac{1}{2}$. 
\end{definition}
The definition of Rademacher complexity we introduced here is the generalization for the finite set case, see \cite{jiao2021error}. By the help of  above notation, we can bound the right-hand side of \eqref{split of statistical error} by the following Lemma.

%We have following lemma to bound the Rademacher complexity of $\omega \cdot \mathcal{F}$, where $\omega$ is a map from $\Omega$ to $\mathbb{R}$.

%\begin{lemma}
    %Assume that $\omega : \Omega \rightarrow \mathbb{R}$ and $|\omega(x)| \leq \mathcal{B}$ for all $x \in \Omega$, then for any function class $\mathcal{F}$, there holds 
    %\begin{equation}
        %\mathfrak{R}_N(\omega \cdot \mathcal{F}) \leq \mathcal{B}\mathfrak{R}_N( \mathcal{F}),
    %\end{equation}
    %where $\omega \cdot \mathcal{F} := {\bar{u}:\bar{u}=\omega(x)u(x),u \in \mathcal{F}} $.
%\end{lemma}

%Proof. See Lemma 5.2 in \cite{jiao2021error}.

%Now we bound the difference between continuous loss $\mathcal{L}_i$ and empirical loss $\widehat{\mathcal{L}}_i$ using Rademacher complexity by the following lemma.

\begin{lemma}\label{lemma: function class Fk}
For $\mathcal{L}_k(u,\bp)$ and $\widehat{\mathcal{L}}_k^N(u,p)$ given in \eqref{define Lk}, we have
\begin{equation*}
			\begin{aligned}
				&\mathbb{E}_{\{\bX_i\}_{i=1}^N} \sup_{u,\bp \in \mathcal{P}} \pm \left[ \mathcal{L}_k(u,\bp) - \widehat{\mathcal{L}}_k^N(u,\bp) \right] 
                \leq C(coe) \mathfrak{R}_N( \mathcal{F}_k),
				%&\mathbb{E}_{\{X_i\}_{i=1}^N} \sup_{u,p \in \mathcal{P}} \pm \left[ \mathcal{L}_n(u,p) - \widehat{\mathcal{L}}_n(u,p) \right] \leq 4 |\Omega| \mathfrak{R}_N( \mathcal{F}_n),\hspace{12pt} n=2,5,6,8,
			\end{aligned}
		\end{equation*}
for $k=1,\dots,8,$ here the function class $\mathcal{F}_k$ is defined by
%\begin{equation*}
    %\mathcal{F}_k = \{ f : f(\bx) = \Gamma_k(u,\bp, \bx),\, u, \bp \in \mathcal{P}\}.
%\end{equation*}
\begin{equation*}
    \begin{aligned}
    \mathcal{F}_1 
    =& \{ |\bp(\bx)|^2 : \bp \in \mathcal{P}\},
    &\mathcal{F}_2 
    =& \{ \bp(\bx) \nabla u(\bx) : u,\bp \in \mathcal{P}\},\\
     \mathcal{F}_3 
    =& \{ |\nabla u(\bx)|^2 : u \in \mathcal{P}\}, 
    &\mathcal{F}_4 
    =& \{ |\divp(\bx)|^2 : \bp \in \mathcal{P}\}, \\
    \mathcal{F}_5 
    =& \{\divp(\bx) u(\bx) : u, \bp \in \mathcal{P}\},
    & \mathcal{F}_6 
    =&  \{ \divp(\bx) : \bp \in \mathcal{P}\}, \\
    \mathcal{F}_7 
    =& \{ u^2(\bx) : u \in \mathcal{P}\},
    & \mathcal{F}_8 
    =& \{ u(\bx) : u \in \mathcal{P}\},
    \end{aligned}
\end{equation*}
\end{lemma}
 
\begin{proof}
In order to prove the Lemma, we recall the important property of Rademacher complexity: given $\omega\in L^\infty(\Omega)$,
%$\omega : \Omega \rightarrow \mathbb{R}$ and $|\omega(x)| \leq \mathcal{B}$ for all $x \in \Omega$, 
then for any function class $\mathcal{F}$, there holds (see Lemma 5.2 in \cite{jiao2021error})
    \begin{equation}\label{property of RC}
        \mathfrak{R}_N(\omega \cdot \mathcal{F}) \leq \Vert \omega \Vert_{L^\infty(\Omega)} \mathfrak{R}_N( \mathcal{F}),
    \end{equation}
    where $\omega \cdot \mathcal{F} := \{\bar{u}:\bar{u} = \omega(\bx)u(\bx),\,u \in \mathcal{F}\} $.

Now let us just prove one case. By taking $\{\widetilde{\bX}_i\}_{i=1}^N$ as an independent copy of $\{\bX_i\}_{i=1}^N$, we have
\begin{equation}
    \begin{aligned}
    \mathcal{L}_7(u,\bp) - \widehat{\mathcal{L}}_7^N(u,\bp) 
    &= 
    |\Omega| \mathbb{E}_{\bX \sim U(\Omega)}\omega^2(\bX) u^2(\bX) 
    - 
    |\Omega|\frac{1}{N} \sum_{k=1}^N \omega^2(\bX_i) u^2(\bX_i)  \\
    &=
    \frac{|\Omega|}{N} \mathbb{E} _{\{\widetilde{\bX}_i \} _{i=1}^N } \sum_{i=1}^N \left[\omega^2(\widetilde{\bX}_i) u^2(\widetilde{\bX}_i)
    - \omega^2(\bX_i) u^2(\bX_i)\right],
    \end{aligned}\nonumber
\end{equation}
Thus we can calculate that
\begin{equation}
    \begin{aligned}
    &\mathbb{E} _{\{\bX_i \} _{i=1}^N } \sup_{u,\bp \in \mathcal{P}} | \mathcal{L}_7(u,\bp) - \widehat{\mathcal{L}}_7^N(u,\bp)| \\
    \leq & 
    \frac{|\Omega|}{N} \mathbb{E} _{\{\bX_i, \widetilde{\bX}_i \} _{i=1}^N } \sup_{u,\bp \in \mathcal{P}} \sum_{i=1}^N 
    \left[\omega^2(\widetilde{\bX}_i) u^2(\widetilde{\bX}_i)
    - \omega^2(\bX_i) u^2(\bX_i)\right]   \\
    =& 
    \frac{|\Omega|}{N} \mathbb{E} _{\{\bX_i, \widetilde{\bX}_i,\sigma_i \} _{i=1}^N } \sup_{u,\bp \in \mathcal{P}} \sum_{i=1}^N \sigma_i 
    \left[\omega^2(\widetilde{\bX}_i) u^2(\widetilde{\bX}_i)
    - \omega^2(\bX_i) u^2(\bX_i)\right],
    \end{aligned}\nonumber
\end{equation}
here the last step is due to the fact that the insertion of Rademacher variables doesn’t change the distribution since the symmetry structure in the summation.
By splitting the supremum into two terms, it continues to have
\begin{equation}
    \begin{aligned}
    \le  &
    \frac{|\Omega|}{N} \mathbb{E} _{\{\widetilde{\bX}_i,\sigma_i \} _{i=1}^N } 
    \sup_{u,\bp \in \mathcal{P}}
    \sum_{i=1}^N \sigma_i \omega^2(\widetilde{\bX}_i) u^2(\widetilde{\bX}_i)\\
    &+ 
    \frac{|\Omega|}{N} \mathbb{E} _{\{\bX_i,\sigma_i \} _{i=1}^N } 
    \sup_{u,\bp \in \mathcal{P}}
    \sum_{i=1}^N -\sigma_i \omega^2(\bX_i) u^2(\bX_i)^2\\
    =& 
    \frac{2|\Omega|}{N} \mathbb{E} _{\{\bX_i,\sigma_i \} _{i=1}^N } 
    \sup_{u,\bp \in \mathcal{P}}
    \sum_{i=1}^N \sigma_i \omega^2(\bX_i) u^2(\bX_i)^2 \\
    = &
    2|\Omega| \mathfrak{R}_N (\omega^2\cdot \mathcal{F}_7)
    \le C(\Omega, coe) \mathfrak{R}_N ( \mathcal{F}_7).
    \end{aligned}
\end{equation}
Here we used the inequality \eqref{property of RC} in the last step.
\end{proof}

%Following we introduce the $\epsilon$-covering number of a set T $\mathcal{C}(\epsilon,T,\tau)$. Here $\tau$ is a metric. Detailed discribtion see definition in 

%\begin{definition}
    %An $\epsilon$-cover of a set $T$ in a metric space $(S,\tau)$ is a subset $T_c\subset S$ such that for each $t\in T$, there exists a $t_c\in T_c$ such that $\tau(t,t_c)\leq\epsilon$. The $\epsilon$-covering number of $T$, denoted as $\mathcal{C}(\epsilon,T,\tau)$ is defined to be the minimum cardinality among all $\epsilon$-cover of $T$ with respect to the metric $\tau$.
%\end{definition}

%In Euclidean space, the $\epsilon$-covering number has an upper bound, seeing the following lemma.

%\begin{lemma}
%Suppose that $T \subset \mathbb{R}^d$ and $\|t\|_2 \leq B$ for $t \in T$,then 
%\begin{equation}\nonumber
    %\mathcal{C}(\epsilon,T,\| \cdot \|_2) \leq \left(\frac{2B\sqrt{d}}{\epsilon} \right)^d.
%\end{equation}
%\end{lemma}

%\begin{proof}
    %See lemma 5.5 in \cite{jiao2021error}.
%\end{proof}
The following Lemma states that the Rademacher complexity can be upper bounded by the Lipschitz constant of the parameter space.
%We can compare the $\epsilon$-covering number of function set and parameter space under some assumptions like that functions are Lipschitz continuous.

\begin{lemma}\label{lemma: bound in terms of eps-cover}
    Let $\mathcal{F}$ be a parameterized class of functions: $\mathcal{F}=\{f(\bx;\theta):\theta \in \Theta \}$. Let
		$||\cdot||_{\Theta}$ be a norm on $\Theta$ and let $||\cdot||_{\mathcal{F}}$ be a norm on $\mathcal{F}$. 
  Suppose that $\mathcal{F}$ temains in a bounded set in $L^\infty(\Omega)$, and the
  mapping $\theta \mapsto f(x;\theta)$ is L-Lipschitz, that is,
		\begin{equation}
  \| f(*;\theta) \|_{L^\infty} \leq B,\qquad
			||f(x;\theta)-f(x;\tilde{\theta})||_{\mathcal{F}} \le L||\theta - \tilde{\theta}||_{\Theta},\nonumber
		\end{equation}
  here the upper bound $B$ is independent of $\theta$ and constant $L$ is independent of $\bx$.
  Then there holds
  \begin{equation}
       \mathfrak{R}_N (\mathcal{F}) \leq \inf_{0 < \delta < B/2} \left(4\delta + \frac{12}{\sqrt{N}} \int_{\delta}^{B/2} \Big\{\log \Big(\frac{2B\sqrt{d}}{\epsilon} \Big)^d \Big\}^{\frac{1}{2}}  d\epsilon \right).
   \end{equation}
		%then for any $\epsilon > 0$,$\mathcal{C}(\epsilon,\mathcal{F},\|\cdot\|_{\mathcal{F}}) \leq 
		%\mathcal{C}(\epsilon /L,\Theta,\|\cdot\|_{\Theta})$.
\end{lemma}
\begin{proof}
    One can find the proof of the Lemma in \cite{jiao2021error}, Lemma 5.5-5.8, which relies on the concept of $\epsilon$-covering set.
\end{proof}

%For any finite set $A \subset \mathbb{R}^N$ with diameter $D=sup_{a \in A} \|a\|_2$, we have 
	%\begin{equation}
		%\mathfrak{R}_N(A) \leq \dfrac{D}{N} \sqrt{2log|A|}.\nonumber
	%\end{equation}
%And to infinite set, we have the other lemma.

%\begin{lemma}
   %Let $\mathcal{F}$ be a function class and $\| f \|_{\infty} \leq B$ for any $f \in \mathcal{F}$, we have 
   %\begin{equation}
       %\mathfrak{R}_N (\mathcal{F}) \leq \inf_{0 < \delta < B/2} \left(4\delta + \frac{12}{\sqrt{N}} \int_{\delta}^{B/2} \sqrt{log \mathcal{C}(\epsilon,\mathcal{F},\| \cdot \|_{\infty}) d\epsilon}  \right).
   %\end{equation}
%\end{lemma}

%\begin{proof}
    %See lemma 5.8 in \cite{jiao2021error}.
%\end{proof}
\subsection{Upper Bound and Lipschitz Constant}
In the following we will derive the upper bound and Lipschitz constant of the parameterized function class.

\begin{lemma}\label{bound and Lipschitz of f}
Set the parameterized function class $\mathcal{P}=\mathcal{N}_{\rho}(\mathcal{D},\mathfrak{n}_\mathcal{D},B_{\theta})$ of network \eqref{network}-\eqref{modified layer}. Under Assumption \ref{assumption: modified function}, for any $\mathbf{f}^{\mathbf{mod}}\in \mathcal{P}$ we have 
\begin{equation*}
    |f_i^{\mathbf{mod}}| \le B_{\theta} + B_{g},
\end{equation*}
and
\begin{equation*}
\begin{aligned}
    & \left|f_i^{\mathbf{mod}}(\bx;\theta)-f_i^{\mathbf{mod}}(\bx;\tilde{\theta})\right| \leq \sqrt{\mathfrak{n}_\mathcal{D}}B_{\theta}^{\mathcal{D}-1} \Bigg( \prod_{j=1}^{\mathcal{D}-1} n_j \Bigg ) \big\|\theta-\tilde{\theta}\big\|_2.
\end{aligned}    
\end{equation*}
\end{lemma}

\begin{proof}
    To compute the Lipschitz constant, we can write directly from the definition of $N_i$ and proposition \ref{assumption: activation function} that
    \begin{equation*}
        \begin{aligned}
            \left|f_i^{(\ell)}-\tilde{f}_i^{(\ell)}\right|
            &= \Big|\rho\Big(\sum_{j=1}^{n_{\ell-1}}a_{ij}^{(\ell)}f_j^{(\ell-1)}+b_i^{(l)}\Big)
            -\rho\Big(\sum_{j=1}^{n_{\ell-1}}\tilde{a}_{ij}^{(\ell)}\tilde{f}_j^{(\ell-1)}-\tilde{b}_i^{(\ell)}\Big)\Big|\\
            &\leq
            \sum_{j=1}^{n_{\ell-1}}\left|a_{ij}^{(\ell)}f_j^{(\ell-1)}-\tilde{a}_{ij}^{(\ell)}\tilde{f}_j^{(\ell-1)}\right|
            +
            \left|b_i^{(\ell)}-\tilde{b}_i^{(\ell)}\right|\\
            &\leq\sum_{j=1}^{n_{\ell-1}}\left|a_{ij}^{(\ell)}\right|\left|f_j^{(\ell-1)}-\tilde{f}_j^{(\ell-1)}\right|\\
            &\qquad +
            \sum_{j=1}^{n_{\ell-1}}\left|a_{ij}^{(\ell)}-\tilde{a}_{ij}^{(\ell)}\right|\left|\tilde{f}_j^{(\ell-1)}\right|+\left|b_i^{(\ell)}-\tilde{b}_i^{(\ell)}\right|,
        \end{aligned}
    \end{equation*}
    summing up the above estimate for $i=1,\dots,n_{\ell}$, we have
    \begin{equation}\label{difference of f}
        \begin{aligned}
            \sum_{i=1}^{n_{\ell}}\left|f_i^{(\ell)}-\tilde{f}_i^{(\ell)}\right|
            &\leq B_{\theta} n_{\ell}
            \sum_{j=1}^{n_{\ell-1}}\left|f_j^{(\ell-1)}-\tilde{f}_j^{(\ell-1)}\right|\\
            & +
            \sum_{i=1}^{n_{\ell}} \Big(\sum_{j=1}^{n_{\ell-1}}\left|a_{ij}^{(\ell)}-\tilde{a}_{ij}^{(\ell)}\right|
            +\left|b_i^{(\ell)}-\tilde{b}_i^{(\ell)}\right|\Big).
        \end{aligned}
    \end{equation}
    Now apply $\ell=1$ to \eqref{difference of f} in which case $\bf^{(0)}=\tilde{\bf}^{(0)}=\bx$, then we can write
    \begin{equation}\label{difference of f_1}
        \begin{aligned}
            \sum_{i=1}^{n_{1}}
            \left|f_i^{(1)}-\tilde{f}_i^{(1)}\right|
            %&= \Big|\rho\Big(\sum_{j=1}^{n_{1}}a_{ij}^{(1)}x_j+b_i^{(1)}\Big)
            %-\rho\Big(\sum_{j=1}^{n_{1}}\tilde{a}_{ij}^{(1)}x_j-\tilde{b}_i^{(1)}\Big)\Big|\\
            &\leq
            \sum_{i=1}^{n_{1}} \Big(
            \sum_{j=1}^{n_{0}}\left|a_{ij}^{(l)}-\tilde{a}_{ij}^{(l)}\right|
            +\left|b_i^{(l)}-\tilde{b}_i^{(l)}\right|\Big)
            =
            \sum_{j=1}^{\mathfrak{n}_1} \left|\theta_j-\tilde{\theta}_j\right|.
        \end{aligned}
    \end{equation}
    Similarly, take $\ell=2$ in \eqref{difference of f} and apply the result in \eqref{difference of f_1}, then one can derive 
    \begin{equation*}
        \begin{aligned}
         \sum_{i=1}^{n_{2}}
            \left|f_i^{(2)}-\tilde{f}_i^{(2)}\right|
            \leq &
            B_{\theta} n_2
            \sum_{j=1}^{\mathfrak{n}_1} \left|\theta_j-\tilde{\theta}_j\right|
            +
             \sum_{i=1}^{n_{2}}\Big(
            \sum_{j=1}^{n_{1}}\left|a_{ij}^{(2)}-\tilde{a}_{ij}^{(2)}\right|\\
            &\qquad\qquad+
            \left|b_i^{(2)}-\tilde{b}_i^{(2)}\right|\Big)
            \le 
            B_{\theta} n_2
            \sum_{j=1}^{\mathfrak{n}_2} \left|\theta_j-\tilde{\theta}_j\right|.
        \end{aligned}
    \end{equation*}
    Iterating above process,
    %\eqref{difference of f}-\eqref{difference of f_2}, 
    we finally obtain for $\ell= 1,\dots \mathcal{D}$ that
    \begin{equation}\label{difference of f_l}
\begin{aligned}
\sum_{i=1}^{n_{\ell}}
    \left|f_i^{(\ell)}-\tilde{f}_i^{(\ell)}\right| 
    &\leq 
    B_{\theta}^{\ell-1} 
    \Bigg( \prod_{j=1}^{\ell} n_j \Bigg ) 
    \sum_{j=1}^{\mathfrak{n}_\ell}\left|\theta_j-\tilde{\theta}_j\right|
    \le 
    \sqrt{\mathfrak{n}_\ell}B_{\theta}^{\ell-1} \Bigg( \prod_{j=1}^{\ell} n_j \Bigg ) \big\|\theta-\tilde{\theta}\big\|_2.
\end{aligned}    
\end{equation}
Taking $\ell = \mathcal{D}$ in above estimate, then the Lemma follows from
Assumption \ref{assumption: modified function}.
%\begin{equation*}
%    f_i^{\mathbf{mod}} \le B_{\phi} + f_i^{(\mathcal{D})},
%    \qquad
 %    f_i^{\mathbf{mod}} 
%    -
%    \widetilde{f}_i^{\mathbf{mod}}
%    \le 
%    f_i^{(\mathcal{D})}-\widetilde{f}_i^{(\mathcal{D})} .
%\end{equation*}
   
\end{proof}

\begin{lemma}\label{bound of f'}
Set the parameterized function class $\mathcal{P}=\mathcal{N}_{\rho}(\mathcal{D},\mathfrak{n}_\mathcal{D},B_{\theta})$ of network \eqref{network}-\eqref{modified layer}. Under Assumption \ref{assumption: modified function}, for any $\mathbf{f}^{\mathbf{mod}}\in \mathcal{P}$ we have 
\begin{equation}\label{estimate of derivative of f^mod}
    \begin{aligned}
        %&|\partial_{\bx_p}f_q^{(\ell)}|\leq \left(\prod_{j=1}^{\ell-1}n_j\right)(B_{\theta}B_{\rho'})^{\ell}, \hspace{6pt}\ell = 1,2,\dots,\mathcal{D}-1,\\
        &|\partial_{\bx_p}f_i^{\mathbf{mod}}(x;\theta)|
        \leq
        \Bigg(\prod_{j=1}^{\mathcal{D}-1}n_j\Bigg)
        B_{\theta}^{\mathcal{D}}
        + B_{g'} .    
     \end{aligned}
\end{equation}
\end{lemma}

\begin{proof}
For $\ell=1,2,\dots,\mathcal{D}-1$, one can write from definition that
\begin{equation}\label{iterate of derivative of f}
    \begin{aligned}
\sum_{i=1}^{n_{\ell}}
\left|\partial_{x_p} f_i^{(\ell)}\right| 
& =
\sum_{i=1}^{n_{\ell}} 
\Bigg|\sum_{j=1}^{n_{\ell-1}} a_{i j}^{(\ell)} 
\partial_{x_p} f_j^{(\ell-1)} \rho^{\prime}
\Big(\sum_{j=1}^{n_{\ell-1}} a_{i j}^{(\ell)} 
f_j^{(\ell-1)}+b_q^{(\ell)}\Big)\Bigg| \\
& \leq 
B_\theta 
\sum_{i=1}^{n_{\ell}} 
\sum_{j=1}^{n_{\ell-1}}
\left|\partial_{x_p} f_j^{(\ell-1)}\right| 
\le 
B_\theta 
n_{\ell} 
\sum_{j=1}^{n_{\ell-1}}
\left|\partial_{x_p} f_j^{(\ell-1)}\right|.
\end{aligned}
\end{equation}
Iterating above process, one can obtain for $\ell= 1,2,\dots \mathcal{D}-1$
\begin{equation}\label{bound of derivative of f^l}
    \begin{aligned}
    \sum_{i=1}^{n_{\ell}} 
|\partial_{x_p}f_i^{(\ell)}|
\leq 
\Bigg(\prod_{j=1}^{\ell}n_j\Bigg)
B_{\theta}^{\ell}.
\end{aligned}
\end{equation}
Taking $\ell = \mathcal{D}-1$ in above estimate, and apply \eqref{iterate of derivative of f} similarly for the estimate of $f_i^{(\mathcal{D})}$, we finally have
\begin{equation*}
    \begin{aligned}
|\partial_{x_p}f_i^{(\mathcal{D})}|
\leq 
\Bigg(\prod_{j=1}^{\mathcal{D}-1}n_j\Bigg)
B_{\theta}^{\mathcal{D}}.
\end{aligned}
\end{equation*}
Under the boundedness condition in Assumption \ref{assumption: modified function} the estimate \eqref{estimate of derivative of f^mod} follows.
%\begin{equation*}
%    \partial_{x_p}f_i^{\mathbf{mod}} \le B_{\phi'} + \partial_{x_p}f_i^{(\mathcal{D})},
%\end{equation*}

\end{proof}

%The next lemma tells us that $\partial_{x_p}f_i^{\mathbf{mod}}$ is also Lipschitz continuous under Assumption 1.

\begin{lemma}\label{Lipschitz of f'}
    Set the parameterized function class $\mathcal{P}=\mathcal{N}_{\rho}(\mathcal{D},\mathfrak{n}_\mathcal{D},B_{\theta})$ of network \eqref{network}-\eqref{modified layer}. Under Assumption \ref{assumption: modified function}, for any $\mathbf{f}^{\mathbf{mod}}\in \mathcal{P}$ we have 
\begin{equation}\label{Lifschitz of derivative of f^mod}
\left|\partial_{x_p}f_i^{\mathbf{mod}} (\bx;\theta)-\partial_{x_p}f_i^{\mathbf{mod}} (\bx;\tilde{\theta})\right|
\leq 
\sqrt{\mathfrak{n}_{\mathcal{D}}}
(\mathcal{D}+2)
B_{\theta}^{2\mathcal{D}-1}\Bigg(\prod_{k=1}^{\mathcal{D}-1}n_k\Bigg)^2\big\|\theta-\tilde{\theta}\big\|_2.
\end{equation}
\end{lemma}

\begin{proof}
From definition we can write for $\ell= 1,\dots \mathcal{D}-1$
\begin{equation*}
    \begin{aligned}
\left|\partial_{x_p} f_i^{(\ell)}-\partial_{x_p} \widetilde{f}_i^{(\ell)}\right| 
= & 
\Big|\sum_{j=1}^{n_{\ell-1}} 
a_{i j}^{(\ell)}\partial_{x_p}f_j^{(\mathcal{\ell}-1)} \rho^{\prime}\Big(\sum_{j=1}^{n_{\ell-1}} 
a_{i j}^{(\ell)} \partial_{x_p}f_j^{(\mathcal{\ell}-1)} 
+
b_i^{(\ell)}\Big)\\
& -
\sum_{j=1}^{n_{\ell-1}}
\widetilde{a}_{i j}^{(\ell)}\partial_{x_p}\tilde{f}_j^{(\mathcal{\ell}-1)}  \rho^{\prime}\Big(\sum_{j=1}^{n_{\ell-1}} \widetilde{a}_{i j}^{(\ell)} 
\partial_{x_p}\tilde{f}_j^{(\mathcal{\ell}-1)} 
+
\widetilde{b}_i^{(\ell)}\Big)\Big|.
\end{aligned}
\end{equation*}
Using triangle inequality, it can be bounded from above by
\begin{equation*}
\begin{gathered}
    %\leq & 
\sum_{j=1}^{n_{\ell-1}}\Big|a_{i j}^{(\ell)}\Big|\left|\partial_{x_p} f_j^{(\ell-1)}\right|\Big|\rho^{\prime}\Big(\sum_{j=1}^{n_{\ell-1}} a_{i j}^{(\ell)} f_j^{(\ell-1)}+b_i^{(\ell)}\Big)-\rho^{\prime}\Big(\sum_{j=1}^{n_{\ell-1}} \widetilde{a}_{i j}^{(\ell)} \widetilde{f}_j^{(\ell-1)}+\widetilde{b}_i^{(\ell)}\Big)\Big| \\
+
\sum_{j=1}^{n_{\ell-1}}\left|a_{i j}^{(\ell)} \partial_{x_p} f_j^{(\ell-1)}-\widetilde{a}_{i j}^{(\ell)} \partial_{x_p} \widetilde{f}_j^{(\ell-1)}\right|\Big|\rho^{\prime}\Big(\sum_{j=1}^{n_{\ell-1}} \widetilde{a}_{i j}^{(\ell)} \widetilde{f}_j^{(\ell-1)}
+
\widetilde{b}_i^{(\ell)}\Big)\Big|,
\end{gathered}
\end{equation*}
thus applying the boundedness and Lipschitz continuity of $\rho'$, it continues to be bounded from above by
\begin{equation*}
\begin{gathered}
%\leq & 
B_\theta \sum_{j=1}^{n_{\ell-1}}\left|\partial_{x_p} f_j^{(\ell-1)}\right|\Big(\sum_{j=1}^{n_{\ell-1}}\left|a_{i j}^{(\ell)} f_j^{(\ell-1)}-\widetilde{a}_{i j}^{(\ell)} \widetilde{f}_j^{(\ell-1)}\right|
+
\left|b_q^{(\ell)}-\widetilde{b}_q^{(\ell)}\right|\Big) \\
+
 \sum_{i=1}^{n_{\ell-1}}\left|a_{i j}^{(\ell)} \partial_{x_p} f_j^{(\ell-1)}-\widetilde{a}_{i j}^{(\ell)} \partial_{x_p} \widetilde{f}_j^{(\ell-1)}\right|.
\end{gathered}
\end{equation*}
Again using the triangle inequality, we can bound above formula by
\begin{equation}\label{iterate of lifschitz estimate of f'}
\begin{gathered}
%\leq & 
B_\theta  \sum_{j=1}^{n_{\ell-1}}\left|\partial_{x_p} f_j^{(\ell-1)}\right|\Big(\sum_{j=1}^{n_{\ell-1}}\left|a_{i j}^{(\ell)}-\widetilde{a}_{i j}^{(\ell)}\right|+B_\theta \sum_{j=1}^{n_{\ell-1}}\left|f_j^{(\ell-1)}-\tilde{f}_j^{(\ell-1)}\right|+\left|b_i^{(\ell)}-\widetilde{b}_i^{(\ell)}\right|\Big) \\
+
 B_\theta \sum_{j=1}^{n_{\ell-1}}\left|\partial_{x_p} f_j^{(\ell-1)}-\partial_{x_p} \widetilde{f}_j^{(\ell-1)}\right|+ \sum_{j=1}^{n_{\ell-1}}\left|a_{i j}^{(\ell)}-\widetilde{a}_{i j}^{(\ell)}\right|\left|\partial_{x_p} \widetilde{f}_j^{(\ell-1)}\right| .
\end{gathered}
\end{equation}
Now applying \eqref{difference of f_l} and \eqref{bound of derivative of f^l} to \eqref{iterate of lifschitz estimate of f'}, the above estimates for $\ell= 2,3,\dots \mathcal{D}-1$ finally  leads to
\begin{equation}\label{iterate of lifschitz estimate of f' final form}
    \begin{aligned}
    \sum_{i=1}^{n_\ell}
    \left|\partial_{x_p} f_i^{(\ell)}-\partial_{x_p} \widetilde{f}_i^{(\ell)}\right| 
        \leq &
        B_\theta n_\ell \sum_{j=1}^{n_{\ell-1}} 
        \left|\partial_{x_p} f_j^{(\ell-1)}-\partial_{x_p} \widetilde{f}_j^{(\ell-1)}\right| \\
        & +
        B_\theta^{2 \ell-1} n_\ell
        \Bigg(\prod_{i=1}^{\ell-1} n_i\Bigg)^2 \sum_{k=1}^{\mathfrak{n}_{\ell}}\left|\theta_k-\widetilde{\theta}_k\right|.
    \end{aligned}
\end{equation}
On the other hand, take $\ell =1$ in \eqref{iterate of lifschitz estimate of f'} and note that $\bf^{(0)}=\tilde{\bf}^{(0)}=\bx$, we can obtain
\begin{equation*}
    \begin{aligned}
\sum_{i=1}^{n_1}
\left|\partial_{x_p} f_i^{(1)}-\partial_{x_p} \widetilde{f}_i^{(1)}\right| 
\leq 
B_\theta 
\sum_{i=1}^{n_1}\Big(\sum_{j=1}^{n_0}
\left|a_{i j}^{(1)}-\widetilde{a}_{i j}^{(1)}\right|
+
\left|b_i^{(1)}-\widetilde{b}_i^{(1)}\right| \Big) 
\leq 
B_\theta \sum_{k=1}^{\mathfrak{n}_1}\left|\theta_k-\widetilde{\theta}_k\right|.
\end{aligned}
\end{equation*}
Starting from above estimate, and iterating the process \eqref{iterate of lifschitz estimate of f' final form}, we can get
\begin{equation}
\sum_{i=1}^{n_{\mathcal{D}-1}}
\left|\partial_{x_p} f_i^{(\mathcal{D}-1)}
-\partial_{x_p} \widetilde{f}_i^{(\mathcal{D}-1)}\right|
\leq \left(\mathcal{D}-1\right)B_{\theta}^{2\mathcal{D}-3}\Bigg(\prod_{k=1}^{\mathcal{D}-1}n_k\Bigg)^2
\sum_{k=1}^{\mathfrak{n}_{\mathcal{D}}}
\left|\theta_k-\widetilde{\theta}_k\right|.
\end{equation}
As for the estimate of $f_i^D$, one can easily derive
\begin{equation*}
\begin{aligned}
\left|\partial_{x_p} f_i^{(\mathcal{D})}-\partial_{x_p} \widetilde{f}_i^{(\mathcal{D})}\right| 
\le & 
B_\theta \sum_{j=1}^{n_{\mathcal{D}-1}}
\left|\partial_{x_p} f_j^{(\mathcal{D}-1)}
-
\partial_{x_p} \widetilde{f}_j^{(\mathcal{D}-1)}\right|
+
 \sum_{j=1}^{n_{\mathcal{D}-1}}\left|a_{i j}^{(\mathcal{D})}
 -
 \widetilde{a}_{i j}^{(\mathcal{D})}\right|\left|\partial_{x_p} \widetilde{f}_j^{(\mathcal{D}-1)}\right| \\
 \le & 
B_\theta \sum_{j=1}^{n_{\mathcal{D}-1}}
\left|\partial_{x_p} f_j^{(\mathcal{D}-1)}
-
\partial_{x_p} \widetilde{f}_j^{(\mathcal{D}-1)}\right|
+
B_{\theta}^{\mathcal{D}-1}
\Bigg(\prod_{j=1}^{\mathcal{D}-1}n_j\Bigg)
 \sum_{k=1}^{\mathfrak{n}_{\mathcal{D}}}
\left|\theta_k-\widetilde{\theta}_k\right|.
 \end{aligned}
\end{equation*}
Under the Lipschitz condition in Assumption \ref{assumption: modified function}, the estimate \eqref{Lifschitz of derivative of f^mod} follows.
%\begin{equation*}
%    \partial_{x_p}f_i^{\mathbf{mod}} 
%    -
%     \partial_{x_p}\widetilde{f}_i^{\mathbf{mod}}
%    \le 
%    \partial_{x_p} f_i^{(\mathcal{D})}-\partial_{x_p} %\widetilde{f}_i^{(\mathcal{D})}
 %   + 2
 %   ( f_i^{(\mathcal{D})}-\widetilde{f}_i^{(\mathcal{D})} ),
%\end{equation*}

\end{proof}

Now let us turn back to the parameterized function space $\mathcal{F}_i$, $i=1,\dots,8$ defined in Lemma \ref{lemma: function class Fk}, it can be concluded from above results that:

\begin{lemma}\label{lemma: bound and L-constant}
    Under Assumption \ref{assumption: modified function} of network \eqref{network}-\eqref{modified layer}, for any function $F_i(\bx;\theta) \in\mathcal{F}_i$, $i=1,\dots,8$, we have
    \begin{equation}
    \begin{aligned}
        |F_i(\bx;\theta)| \leq B_i,\quad
        |F_i(\bx;\theta)-F_i(\bx;\tilde{\theta})|&\leq L_i\|\theta-\tilde{\theta}\|_2,
    \end{aligned}\nonumber
\end{equation}
for any $x\in\Omega$, with 
\begin{equation*}
    B_i \le 4d \Bigg(\prod_{j=1}^{\mathcal{D}-1}n_j\Bigg)^2
        B_{\theta}^{2\mathcal{D}},
    \qquad 
    L_i \le  4d \sqrt{\mathfrak{n}_{\mathcal{D}}} (\mathcal{D}+2)
B_{\theta}^{3\mathcal{D}-1}\Bigg(\prod_{j=1}^{\mathcal{D}-1}n_j\Bigg)^3,
\end{equation*}
for all $i=1,\dots,8$.

\end{lemma}

\begin{proof}
Let us denote the upper bound and Lipschitz constant of $f_i^{\mathbf{mod}}$ in Lemma \ref{bound and Lipschitz of f} by 
\begin{equation*}
    B_f = B_{\theta} + B_{\phi},
    \qquad 
    L_f = \sqrt{\mathfrak{n}_\mathcal{D}}B_{\theta}^{\mathcal{D}-1} \Bigg( \prod_{j=1}^{\mathcal{D}-1} n_j \Bigg ).
\end{equation*}
Furthermore, we can denote the upper bound and Lipschitz constant of $\partial_{x_p}f_i^{\mathbf{mod}}$ in Lemma \ref{bound of f'}-\ref{Lipschitz of f'} by 
\begin{equation*}
    B_{f'} = \Bigg(\prod_{j=1}^{\mathcal{D}-1}n_j\Bigg)
        B_{\theta}^{\mathcal{D}} + B_{\phi'},
    \qquad 
    L_{f'} = \sqrt{\mathfrak{n}_{\mathcal{D}}}
(\mathcal{D}+2)
B_{\theta}^{2\mathcal{D}-1}\Bigg(\prod_{j=1}^{\mathcal{D}-1}n_j\Bigg)^2.
\end{equation*}
Then combining the previous results, one can directly obtain that 
\begin{equation}
    \left\{\begin{aligned}
        &B_1 = dB_f^2,
        & &B_2= B_5 = d B_fB_{f'}, \\
        &B_3 = B_4 = d B_{f'}^2,
        & &B_6 = d B_{f'},\quad
        B_7=B_8^2= B_f^2,
    \end{aligned}\right.\nonumber
\end{equation}
and
\begin{equation}
    \left\{\begin{aligned}
        &L_1 = dL_7 = 2d B_f L_f, &
        &L_2 = L_5 = 2d B_f L_{f'} + 2dB_{f'}L_f, \\
        &L_3 = L_4 = 2dB_{f'}L_{f'},&
        &L_6=d L_{f'},\quad
        L_8 = L_f.
    \end{aligned}\right.\nonumber
\end{equation}
Here, we only show the calculation of $L_1$ and $B_1$, the other results can be derived similarly.
For any parameterized function $F_1(\bx;\theta)\in \mathcal{F}_1=\{|\bp|^2:(u,\bp)\in\mathcal{P}\}$, it follows directly from Lemma \ref{bound and Lipschitz of f} that
\begin{equation*}
    F_1(\bx;\theta) = |\bp|^2 
    = \sum_{i=2}^{d+1} ( f_i^{\mathbf{mod}})^2 
    \le dB_f^2.
\end{equation*}
Note that we also have
\begin{equation*}
\begin{aligned}
    | F_1(\bx;\theta) -  F_1(\bx;\tilde{\theta}) |
    =
    | |\bp(\bx;\theta)|^2 -  |\bp(\bx;\tilde{\theta})|^2 |
    =
     | \sum_{i=1}^d (f_i^{\mathbf{mod}})^2 
     -  \sum_{i=1}^d  (\tilde{f}_i^{\mathbf{mod}})^2 |,
     \end{aligned}
\end{equation*}
thus one can derive from Lemma \ref{bound and Lipschitz of f} that
\begin{equation*}
\begin{aligned}
    | F_1(\bx;\theta) -  F_1(\bx;\tilde{\theta}) |
    \le &
    \sum_{i=2}^{d+1}  | (f_i^{\mathbf{mod}}+ \tilde{f}_i^{\mathbf{mod}})
    \cdot 
    (f_i^{\mathbf{mod}}- \tilde{f}_i^{\mathbf{mod}}) |
    \le 
    2d B_f
    L_f \big\|\theta-\tilde{\theta}\big\|_2.
     \end{aligned}
\end{equation*}
The evaluation of $L_1$ and $B_1$ is obtained.

\end{proof}

\subsection{Proof of Theorem \ref{thm: Statistical Error}}
Now combining the above results, we finally obtain the estimate of statistical error $\mathbb{E}_{\{\bX_i\}_{i=1}^N}\sup_{u\in\mathcal{P}} \pm [\mathcal{L}(u)- \widehat{\mathcal{L}}(u)]$  in Theorem \ref{thm: Statistical Error}. In fact,
    directly from Lemma \ref{lemma: bound in terms of eps-cover} we have
    \begin{equation}\nonumber
        \begin{aligned}
            \mathfrak{R}_N(\mathcal{F}_i)
            &\leq 
            \inf_{0\le\delta\le B_i/2}\left(4\delta+
            \frac{12\sqrt{\mathfrak{n}_{\mathcal{D}}}}{\sqrt{N}}\int_{\delta}^{B_i/2}\sqrt{
            \log\left(\frac{2B_{\theta}L_i\sqrt{\mathfrak{n}_{\mathcal{D}}}}{\epsilon}\right)}d\epsilon\right)\\
            &\leq 
            \inf_{0\le\delta\le B_i/2}\left(4\delta+\frac{6\sqrt{\mathfrak{n}_{\mathcal{D}}}B_i}{\sqrt{N}}\sqrt{\log\left(\frac{2L_iB_{\theta}\sqrt{\mathfrak{n}_{\mathcal{D}}}}{\delta}\right)}\right).
        \end{aligned}
    \end{equation}
    Choosing $\delta=1/\sqrt{N}\le B_i/2$ and applying Lemma \ref{lemma: bound and L-constant}, it leads to
    \begin{equation}
        \begin{aligned}
\mathfrak{R}_N\left(\mathcal{F}_i\right) 
& \leq 
\frac{4}{\sqrt{N}}
+
\frac{6 \sqrt{\mathfrak{n}_{\mathcal{D}}} B_i}{\sqrt{N}} 
\sqrt{\log \left(2 L_i B_\theta \sqrt{\mathfrak{n}_{\mathcal{D}}} \sqrt{N}\right)} \\
& \leq 
C \frac{d \sqrt{\mathfrak{n}_{\mathcal{D}}}
\left(\prod_{i=1}^{\mathcal{D}-1} n_i\right)^2 B_\theta^{2 \mathcal{D}}}{\sqrt{N}} \sqrt{\log \left(8 d \mathfrak{n}_{\mathcal{D}}(\mathcal{D}+2) B_\theta^{3 \mathcal{D}-1}\left(\prod_{i=1}^{\mathcal{D}-1} n_i\right)^3 \sqrt{N}\right)} \\
& \leq C \frac{d \sqrt{\mathcal{D}} \mathfrak{n}_{\mathcal{D}}^{2 \mathcal{D}} B_\theta^{2 \mathcal{D}}}{\sqrt{N}} \sqrt{\log \left(d \mathcal{D} \mathfrak{n}_{\mathcal{D}} B_\theta N\right)}
\end{aligned}
    \end{equation}
    Then it follows from Lemma \ref{lemma: function class Fk} and \eqref{split of statistical error} that
    \begin{equation}
        \mathbb{E}_{\left\{\bX_i\right\}_{i=1}^N} \sup _{u \in \mathcal{P}} \pm[\mathcal{L}(u)-\widehat{\mathcal{L}}(u)] 
        \leq 
        C(\operatorname{coe}) \frac{d \sqrt{\mathcal{D}} \mathfrak{n}_{\mathcal{D}}^{2 \mathcal{D}} B_\theta^{2 \mathcal{D}}}{\sqrt{N}} \sqrt{\log \left(d \mathcal{D} \mathfrak{n}_{\mathcal{D}} B_\theta N\right)}.
    \end{equation}

\section{Conclusions and Extensions}\label{sec: conclusion}

In this paper, we analyze the convergence rate for the mixed residual method for second-order elliptic equations with Dirichlet, Neumann, and Robin boundary conditions respectively. Additionally, we give some estimates for the settings of depth and width of neural networks such that a given convergence rate can be achieved in terms of training samples. To be clear, the convergence rate with respect to $n$ we deduced is $\mathcal{O}(n^{-\frac{1}{dlogd}})$.

Based on the comparison in Table \ref{tab}, it can be concluded that MIM has an advantage over DRM and DGMW for the Dirichlet case due to its ability to enforce the boundary condition. However, for the Neumann and Robin cases, MIM shows similar results to the other methods. While MIM has shown better numerical approximation in most of the experiments conducted, it does not seem to be observable in our analysis. This could be attributed to the lack of analysis of the optimization error, which contains information about the iteration times. Additionally, it should be noted that MIM requires more regularity for the exact solution, as the strong form of the loss function in the least-squares sense has been used.

%Here, we also leave some further research directions. First, as described in the paper, the network structure of MIM is different from a classical neural network, thus, whether UAT can be directly applied might deserve detailed discussion. Second, as shown in Lemma 4.6, we compute eight pairs of $B_i$ and $L_i$, 16 items namely, for the approximation of the statistical error. And in Theorem 3.2, we need to give a shared upper bound of the all 16 items to do the further computation. As these items are a little complex, the upper bound we give might not be a sharp one. As a result, the final approximation we give might be rough, which could be of some negative influences to the final conclusion, although it should not change the scaling.

\section{Acknowledgements}
We would like to thank Professor Jingrun Chen (USTC) for helpful guidance.

The work of Kai Gu and Peng Fang is supported by Undergraduate Training Program for Innovation and Entrepreneurship, Soochow University Project 202110285019Z. The work of Zhiwei Sun is supported by China Scholarship Council via grant 202006920083.
The work of Rui Du is supproted by NSFC grants 12271360 and 11501399.

\bibliographystyle{unsrt}
\bibliography{reference}

\begin{thebibliography}{10}

\bibitem{beck2019machine}
Christian Beck, Weinan E, and Arnulf Jentzen.
\newblock Machine learning approximation algorithms for high-dimensional fully
  nonlinear partial differential equations and second-order backward stochastic
  differential equations.
\newblock {\em Journal of Nonlinear Science}, 29:1563--1619, 2019.

\bibitem{hutzenthaler2020overcoming}
Martin Hutzenthaler, Arnulf Jentzen, and von~Wurstemberger Wurstemberger.
\newblock Overcoming the curse of dimensionality in the approximative pricing
  of financial derivatives with default risks.
\newblock 2020.

\bibitem{han2017deep}
Weinan E, Jiequn Han, and Arnulf Jentzen.
\newblock Deep learning-based numerical methods for high-dimensional parabolic
  partial differential equations and backward stochastic differential
  equations.
\newblock {\em Communications in mathematics and statistics}, 5(4):349--380,
  2017.

\bibitem{yu2018deep}
Weinan E and Bing Yu.
\newblock The deep {Ritz} method: a deep learning-based numerical algorithm for
  solving variational problems.
\newblock {\em Communications in Mathematics and Statistics}, 6(1):1--12, 2018.

\bibitem{han2018solving}
Jiequn Han, Arnulf Jentzen, and Weinan E.
\newblock Solving high-dimensional partial differential equations using deep
  learning.
\newblock {\em Proceedings of the National Academy of Sciences},
  115(34):8505--8510, 2018.

\bibitem{chen2019quasi}
Jingrun Chen, Rui Du, Panchi Li, and Liyao Lyu.
\newblock Quasi-monte carlo sampling for solving partial differential equations
  by deep neural networks.
\newblock {\em Numerical Mathematics: Theory, Methods and Applications},
  14(2):377--404, 2021.

\bibitem{duan2021analysis}
Chenguang Duan, Yuling Jiao, Yanming Lai, Xiliang Lu, Qimeng Quan, and
  Jerry~Zhijian Yang.
\newblock Analysis of deep {Ritz} methods for {Laplace} equations with
  {Dirichlet} boundary conditions.
\newblock arXiv:2111.02009(2021).

\bibitem{sirignano2018dgm}
Justin Sirignano and Konstantinos Spiliopoulos.
\newblock {DGM}: A deep learning algorithm for solving partial differential
  equations.
\newblock {\em Journal of computational physics}, 375:1339--1364, 2018.

\bibitem{raissi2019physics}
Maziar Raissi, Paris Perdikaris, and George~E Karniadakis.
\newblock Physics-informed neural networks: A deep learning framework for
  solving forward and inverse problems involving nonlinear partial differential
  equations.
\newblock {\em Journal of Computational physics}, 378:686--707, 2019.

\bibitem{zang2020weak}
Yaohua Zang, Gang Bao, Xiaojing Ye, and Haomin Zhou.
\newblock Weak adversarial networks for high-dimensional partial differential
  equations.
\newblock {\em Journal of Computational Physics}, 411:109409, 2020.

\bibitem{CAI2020109707}
Zhiqiang Cai, Jingshuang Chen, Min Liu, and Xinyu Liu.
\newblock Deep least-squares methods: An unsupervised learning-based numerical
  method for solving elliptic {PDEs}.
\newblock {\em Journal of Computational Physics}, 420:109707, 2020.

\bibitem{CAI2021110514}
Zhiqiang Cai, Jingshuang Chen, and Min Liu.
\newblock Least-squares {ReLU} neural network ({LSNN}) method for linear
  advection-reaction equation.
\newblock {\em Journal of Computational Physics}, 443:110514, 2021.

\bibitem{yang2021local}
Jiang Yang and Quanhui Zhu.
\newblock A local deep learning method for solving high order partial
  differential equations.
\newblock {\em Numerical Mathematics-Theory Methods and Applications}, 2021.

\bibitem{lyu2022mim}
Liyao Lyu, Zhen Zhang, Minxin Chen, and Jingrun Chen.
\newblock {MIM}: A deep mixed residual method for solving high-order partial
  differential equations.
\newblock {\em Journal of Computational Physics}, 452:110930, 2022.

\bibitem{CSIAM-AM-2-748}
Liyao Lyu, Keke Wu, Rui Du, and Jingrun Chen.
\newblock Enforcing exact boundary and initial conditions in the deep mixed
  residual method.
\newblock {\em CSIAM Transactions on Applied Mathematics}, 2(4):748--775, 2021.

\bibitem{li2022priori}
Lingfeng Li, Xue-cheng Tai, Jiang Yang, and Quanhui Zhu.
\newblock Priori error estimate of deep mixed residual method for elliptic
  {PDEs}.
\newblock arXiv:2206.07474(2022).

\bibitem{jiao2021error}
Yuling Jiao, Yanming Lai, Yisu Lo, Yang Wang, and Yunfei Yang.
\newblock Error analysis of deep {Ritz} methods for elliptic equations.
\newblock arXiv:2107.14478(2021).

\bibitem{jiao_convergence_2023}
Yuling Jiao, Yanming Lai, Yang Wang, Haizhao Yang, and Yunfei Yang.
\newblock Convergence {Analysis} of the {Deep} {Galerkin} {Method} for {Weak}
  {Solutions}, February 2023.
\newblock arXiv:2302.02405(2023).

\bibitem{guhring2021approximation}
Ingo G{\"u}hring and Mones Raslan.
\newblock Approximation rates for neural networks with encodable weights in
  smoothness spaces.
\newblock {\em Neural Networks}, 134:107--130, 2021.

\end{thebibliography}

\end{document}